\documentclass[a4paper,10pt]{article}
\usepackage{amsmath,amssymb,amscd,amsthm}
\usepackage{epic,eepic}
\usepackage[dvips]{graphics}
\DeclareFontEncoding{OT2}{}{}
\DeclareFontSubstitution{OT2}{wncyr}{m}{n}
\DeclareSymbolFont{cyss}{OT2}{wncyss}{m}{n}
\DeclareSymbolFont{cyr}{OT2}{wncyr}{m}{n}
\DeclareMathSymbol{\sh}{\mathbin}{cyss}{`x}

\newcommand{\Li}{\operatorname{Li}}

\newcommand{\C}{{\mathbf C}}
\newcommand{\bP}{{\mathbf P}}
\newcommand{\cL}{{\mathcal L}}
\newcommand{\cM}{{\mathcal M}}

\newcommand{\ds}{\displaystyle}

\newtheorem{thm}{Theorem}

\newtheorem{prop}[thm]{Proposition}

\title{The inversion formula of polylogarithms and \\ the Riemann-Hilbert problem}
\author{OI, Shu and UENO, Kimio}
\date{}

\begin{document}
\maketitle
\insert\footins{\footnotesize 2010 {\it Mathematics Subject Classification.} Primary 34M50,11G55; Secondary 30E25,11M06,32G34;}

\begin{abstract}
In this article, we set up a method of reconstructing the polylogarithms $\Li_k(z)$ from zeta values $\zeta(k)$ via the Riemann-Hilbert problem. This is referred to as ``a recursive Riemann-Hilbert problem of additive type.'' Moreover, we suggest a framework of interpreting the connection problem of the Knizhnik-Zamolodochikov equation of one variable as a Riemann-Hilbert problem.
\end{abstract}

\section{Introduction}
Polylogarithms $\Li_k(z)$\, $(k\ge 2)$ satisfy {\bf the inversion formula}
\begin{equation*}
\Li_k(z)+\sum_{j=1}^{k-1}\frac{(-1)^j\log^j z}{j!}\Li_{k-j}(z)+\Li_{2,\underbrace{\scriptstyle 1,\ldots,1}_{k-2}}(1-z)=\zeta(k).
\end{equation*}
Applying the Riemann-Hilbert problem of additive type (alternatively, Plemelj-Birkhoff decomposition) \cite{Bi, Mu, Pl} to this inversion formula, we show that $\Li_k(z)$ can be reconstructed from boundary values $\zeta(k)$. We prove this by using the Riemann-Hilbert problem recursively so that we refer to this method as \text{\bf a recursive Riemann-Hilbert problem of additive type}.

As a generalization of this method, we can reconstruct multiple polylogarithms $\Li_{k_1,\ldots,k_r}(z)$ from multiple zeta values $\zeta(k_1,\ldots,k_r)$. This is nothing but interpreting {\bf the connection relation}\cite{OiU}
\begin{align*}
\cL(z)=\cL^{(1)}(z)\, \varPhi_{\mathrm KZ}
\end{align*}
between the fundamental solutions of {\bf the Knizhnik-Zamolodochikov equation of one variable} (KZ equation, for short)
\begin{align*}
\frac{dG}{dz}= \left( \frac{X_0}{z}+\frac{X_1}{1-z}\right) G
\end{align*}
as a Riemann-Hilbert problem. Here $\varPhi_{\mathrm KZ}$ is {\bf Drinfel'd associator} and $\cL(z)$ (resp. $\cL^{(1)}(z)$) is the fundamental solution of KZ equation normalized at $z=0$ (resp. $z=1$). We have completely solved this problem and a preprint is now in preparation.

\paragraph{\bf Acknowledgment}

The first author is supported by Waseda University Grant for Special Research Projects No. 2011B-095. The second author is partially supported by JSPS Grant-in-Aid No. 22540035.

\section{The inversion formula of polylogarithms}

For  positive integers $k$, polylogarithms $\Li_k(z)$ are introduced as follows: 
First we set $\Li_1(z)=-\log(1-z)$. In the domain $D=\C\setminus \{z=x\;|\; 1\le x\}$, 
$\Li_1(z)$ has a branch such that $\Li_1(0)=1$ (the principal value of $\Li_1(z)$). 
Starting from the principal value of $\Li_1(z)$, we introduce $\Li_k(z)$, 
which are holomorphic on $D$, recursively by
\begin{equation}
\Li_k(z)=\int_0^z\frac{\Li_{k-1}(t)}{t}dt \qquad (k\ge 2). \label{def:PL}
\end{equation}
where the integral contour is assumed to be in $D$. Then $\Li_k(z)$ has a Taylor expansion
\begin{equation}
\Li_k(z)=\sum_{n=1}^\infty \frac{z^n}{n^k}
\end{equation}
on $|z|<1$. We obtain, for $k\ge 2$, 
\begin{equation}
\lim_{z \to 1,z \in D}\Li_k(z)=\zeta(k), \label{PL_zeta}
\end{equation}
where $\zeta(k)$ is the Riemann zeta value $\ds \zeta(k)=\sum_{n=1}^\infty \frac{1}{n^k}$.

From \eqref{def:PL}, we have differential recursive relations:
\begin{equation}
\frac{d}{dz}\Li_1(z)=\frac{1}{1-z},\qquad \frac{d}{dz}\Li_k(z)=\frac{\Li_{k-1}(z)}{z} \qquad (k \ge 2).
\end{equation}

By virtue of \eqref{def:PL}, $\Li_k(z)$ is analytically continued to a many-valued analytic function on $\bP^1\setminus\{0,1,\infty\}$. However, in this article, we will use the notation $\Li_k(z)$ as the principal value stated previously.\\	

We also define multiple polylogarithms $\Li_{2,1,\ldots,1}(z)$ ($k\ge 2$) as
\begin{align}
\Li_{2,\underbrace{\scriptstyle 1,\ldots,1}_{k-2}}(z)&=\int_0^z\frac{(-1)^{k-1}}{(k-1)!}\frac{\log^{k-1}(1-t)}{t}dt. 
            \label{def:MPL21...1}
\end{align}

By using these relations and \eqref{PL_zeta}, one can obtain easily {\bf the inversion formula} of polylogarithms.
\begin{prop}[the inversion formula of polylogarithms]
For $k\ge 2$, the following functional relation holds.
\begin{equation}
\Li_k(z)+\sum_{j=1}^{k-1}\frac{(-1)^j\log^j z}{j!}\Li_{k-j}(z)+\Li_{2,\underbrace{\scriptstyle 1,\ldots,1}_{k-2}}(1-z)
=\zeta(k). \label{inversion01}
\end{equation}
\end{prop}

\begin{proof}
Differentiating the left hand side of the equation \eqref{inversion01}, we have
\begin{align*}
&\frac{d}{dz}\left(\Li_k(z)+\sum_{j=1}^{k-1}\frac{(-1)^j\log^j z}{j!}\Li_{k-j}(z)+\Li_{2,\underbrace{\scriptstyle 1,\ldots,1}_{k-2}}(1-z)\right)=0.
\end{align*}
Therefore the left hand side of \eqref{inversion01} is a constant. Taking the limit of the left hand side of \eqref{inversion01} as $z \in D$ tends to $1$ and using \eqref{PL_zeta}, we see that the constant is equal to $\zeta(k)$.
\end{proof}

The branch of $\Li_{2,\underbrace{\scriptstyle 1,\ldots,1}_{k-2}}(1-z)$ on the domain $D'=\C\setminus\{z=x\;|\; x\le 0\}$ is determined from the principal value of $\log z$.\\

\section{The recursive Riemann-Hilbert problem of additive type}

Let $D^{(+)}, D^{(-)}$ be domains of $\C$ defined by
\begin{align*}
D^{(+)}&=\{z=x+yi\;|\; x<1,\; -\infty<y<\infty\} \subset D,\\
D^{(-)}&=\{z=x+yi\;|\; 0<x,\; -\infty<y<\infty\} \subset D'.
\end{align*}

\begin{figure}
\begin{picture}(0,3.5)(0,0)
\put(0,0){\includegraphics{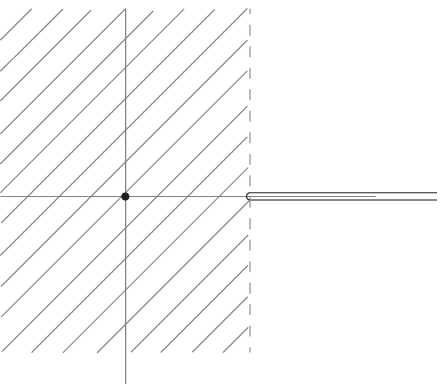}}
\put(0.5,3.5){$D^{(+)}$}
\put(1.1,1.55){0}
\put(2.55,1.55){1}
\put(6,0){\includegraphics{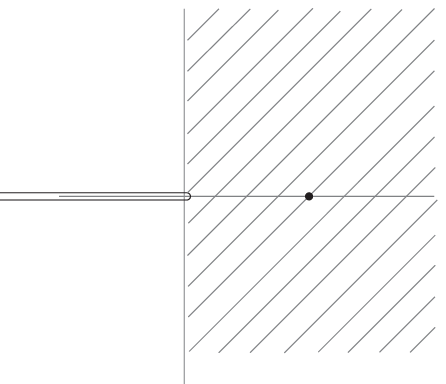}}
\put(9.5,3.5){$D^{(-)}$}
\put(7.6,1.55){0}
\put(9.05,1.55){1}
\end{picture}
\caption{The domains $D^{(+)}, D^{(-)}$.}
\end{figure}

The following theorem says that polylogarithms $\Li_k(z)$ are characterized by the inversion formula.

\begin{thm}
Put $f_1^{(+)}(z)=\Li_1(z)$. For $k\ge 2$, we assume that $f_k^{(\pm)}(z)$ are holomorphic functions on 
$D^{(\pm)}$ satisfying the functional relation
\begin{equation}
f_k^{(+)}(z)+\sum_{j=1}^{k-1}\frac{(-1)^j\log^j z}{j!}f_{k-j}^{(+)}(z)+f_k^{(-)}(z)=\zeta(k) 
\quad (z \in D^{(+)}\cap D^{(-)}), \label{PL_RH01}
\end{equation}
the asymptotic conditions
\begin{equation}
\frac{d}{dz}f_k^{(\pm)}(z) \to 0 \quad (z \to \infty,\; z \in D^{(\pm)}), \label{PL_RH01_asym}
\end{equation}
and the normalization condition
\begin{equation}
f_k^{(+)}(0)=0. \label{PL_RH01_norm}
\end{equation}
Then we have
\begin{equation*}
f_k^{(+)}(z)=\Li_{k}(z),\qquad f_k^{(-)}(z)=\Li_{2,\underbrace{\scriptstyle 1,\ldots,1}_{k-2}}(1-z) \qquad (k \ge 2).
\end{equation*}
\end{thm}

\begin{proof}
We prove the theorem by induction on $k \ge 2$. For the case $k=2$, 
the proof can be done in the same manner as the case $k>2$ from the definition of $f^{(+)}_1(z)$. 
So we assume that 
$f_j^{(+)}(z)=\Li_j(z)$ 
and $f_j^{(-)}(z)=\Li_{2,\underbrace{\scriptstyle 1,\ldots,1}_{j-2}}(1-z)$ for $2\le j \le k-1$. 
Now we show that $f_k^{(+)}(z)=\Li_k(z)$, $f_k^{(-)}(z)=\Li_{2,\underbrace{\scriptstyle 1,\ldots,1}_{k-2}}(1-z)$. 
From the assumption, the equation \eqref{PL_RH01} becomes
\begin{equation}
f_k^{(+)}(z)+\sum_{j=1}^{k-1}\frac{(-1)^j\log^j z}{j!}\Li_{k-j}(z)+f_k^{(-)}(z)=\zeta(k). \label{PL_RH01_proof1}
\end{equation}
Differentiating this equation, we have
\begin{align*}
0&=\frac{d}{dz}\left(f_k^{(+)}(z)+\sum_{j=1}^{k-1}\frac{(-1)^j\log^j z}{j!}\Li_{k-j}(z)+f_k^{(-)}(z)\right)\\
&=\frac{d}{dz}f_k^{(+)}(z)+\sum_{j=1}^{k-2}\left(
\frac{1}{z}\frac{(-1)^j\log^{j-1} z}{(j-1)!}\Li_{k-j}(z)
+\frac{(-1)^j\log^j z}{j!}\frac{\Li_{k-j-1}(z)}{z}
\right)\\
&\qquad +\frac{1}{z}\frac{(-1)^{k-1}\log^{k-2} z}{(k-2)!}\Li_{1}(z)
+\frac{(-1)^{k-1}\log^{k-1} z}{(k-1)!}\frac{1}{1-z}\\
&\qquad +\frac{d}{dz}f_k^{(-)}(z)\\
&=\frac{d}{dz}f_k^{(+)}(z)-\frac{\Li_{k-1}(z)}{z}
+\frac{1}{1-z}\frac{(-1)^{k-1}\log^{k-1} z}{(k-1)!}+\frac{d}{dz}f_k^{(-)}(z).
\end{align*}
Thus we obtain
\begin{equation}
\frac{d}{dz}f_k^{(+)}(z)-\frac{\Li_{k-1}(z)}{z}=-\frac{1}{1-z}\frac{(-1)^{k-1}\log^{k-1} z}{(k-1)!}
-\frac{d}{dz}f_k^{(-)}(z) \label{PL_RH01_proof2}
\end{equation}
on $z \in D^{(+)}\cap D^{(-)}$. Here, the left hand side of \eqref{PL_RH01_proof2} is holomorphic on $D^{(+)}$ 
and the right hand side of \eqref{PL_RH01_proof2} is holomorphic on $D^{(-)}$. 
Therefore the both sides of \eqref{PL_RH01_proof2} are entire functions. 
Using the asymptotic condition \eqref{PL_RH01_asym} and
\begin{equation*}
\frac{\Li_{k-1}(z)}{z}\to 0 \quad (z \to \infty, z \in D^{(+)}),\qquad \frac{\log^{k-1} z}{1-z}\to 0 
\quad (z \to \infty, z \in D^{(-)}),
\end{equation*}
we have that both sides of \eqref{PL_RH01_proof2} are 0 by virtue of Liouville's theorem. Therefore we have
\begin{align*}
f_k^{(+)}(z)&=\int^z\frac{\Li_{k-1}(z)}{z}dz=\Li_{k}(z)+c^{(+)}_k,\\
f_k^{(-)}(z)&=\int^z-\frac{1}{1-z}\frac{(-1)^{k-1}\log^{k-1} z}{(k-1)!}dz
=\Li_{2,\underbrace{\scriptstyle 1,\ldots,1}_{k-2}}(1-z)+c^{(-)}_k,
\end{align*}
where $c^{(+)}_k, c^{(-)}_k$ are integral constants. 
From the normalization condition \eqref{PL_RH01_norm}, it is clear that $c^{(+)}_k$ is equal to $0$. 
Finally, substituting $f_k^{(+)}(z)$ and $f_k^{(-)}(z)$ in \eqref{PL_RH01}, we obtain
\begin{equation}
\Li_{k}(z)+\sum_{j=1}^{k-1}\frac{(-1)^j\log^j z}{j!}\Li_{k-j}(z)+
\Li_{2,\underbrace{\scriptstyle 1,\ldots,1}_{k-2}}(1-z)+c^{(-)}_k=\zeta(k).
\end{equation}
Comparing the inversion formula \eqref{inversion01}, we have $c^{(-)}_k=0$. This concludes the proof.
\end{proof}

The equation \eqref{PL_RH01_proof1} is interpreted as the decomposition of the holomorphic function
\begin{equation*}
\sum_{j=1}^{k-1}\frac{(-1)^j\log^j z}{j!}\Li_{k-j}(z)
\end{equation*}
on $z \in D^{(+)}\cap D^{(-)}$ to a sum of a function $f_k^{(+)}(z)$, which is holomorphic on 
$D^{(+)}$, and a function $f_k^{(-)}(z)$, which is holomorphic on $D^{(-)}$. 
This decomposition is nothing but a Riemann-Hilbert problem of additive type. 
The theorem says that polylogarithms $\Li_k(z)$ can be constructed from the boundary value $\zeta(k)$ 
by applying this Riemann-Hilbert problem recursively. 
In this sense, we call \eqref{PL_RH01} {\bf the recursive Riemann-Hilbert problem of additive type}.

\vspace{1cm}

{\noindent
{\bf OI, Shu.}\\
Department of Mathematics, School of Fundamental Sciences and Engineering, Faculty of Science and Engineering, Waseda university. 3-4-1, Okubo, Shinjuku-ku, Tokyo 169-8555, Japan.\\
{\it e-mail:} {\tt shu\_oi@toki.waseda.jp}\\[1\baselineskip]
{\bf UENO, Kimio}\\
Department of Mathematics, School of Fundamental Sciences and Engineering, Faculty of Science and Engineering, Waseda university. 3-4-1, Okubo, Shinjuku-ku, Tokyo 169-8555, Japan.\\
{\it e-mail:} {\tt uenoki@waseda.jp}


\begin{thebibliography}{BBBL}

\bibitem[{\bf Bi}]{Bi}
{G. D. Birkhoff},
{The generalized Riemann problem for linear differential equations and the allied problems for linear difference and q-difference equations},
Proc. Am. Acad. Arts and Sciences, {\bf 49} (1914), 521-568.

\bibitem[{\bf Mu}]{Mu}
{N. I. Muskhelishvili}, 
{Singular Integral Equations},
P. Noordhoff Ltd. (1946). 


\bibitem[{\bf OiU}]{OiU}
{S. Oi and K. Ueno},
{Connection Problem of Knizhnik-Zamolodchikov Equation on Moduli Space $\cM_{0,5}$ },
preprint (2011) arXiv:1109.0715. 

\bibitem[{\bf Pl}]{Pl}
{J. Plemelj},
{Problems in the sense of Riemann and Klein},
Interscience Tracts in Pure and Applied Mathematics, No. 16, Interscience Publishers, John Wiley \& Sons Inc. New York-London-Sydney (1964).


\end{thebibliography}
\end{document}